\documentclass[a4paper, 11pt, reqno,amsmath,amsthm,amssymb,amscd]{amsart}

\usepackage{mathrsfs,amssymb, amscd,amsmath,amsthm}
\usepackage[enableskew,vcentermath]{youngtab}
\usepackage{multicol}\multicolsep=0pt

\def\crulefill{\leavevmode\leaders\hrule height 1pt\hfill\kern 0pt}
\long\def\QUERY#1{%
\leavevmode\newline%
\noindent$\star\star\star$\thinspace\textsf{Comment/Query}\crulefill\newline%
   \space #1\newline\hbox to 120mm{\crulefill}$\star\star\star$\newline
}

\numberwithin{equation}{section} \theoremstyle{definition}
\newtheorem{Defn}[equation]{Definition}

\theoremstyle{plain}
\newtheorem{Prop}[equation]{Proposition}
\newtheorem{Theorem}[equation]{Theorem}

\newtheorem{Lemma}[equation]{Lemma}

\def\Case#1{\medskip\noindent\textbf{Case #1}:\leavevmode\newline}

\makeatletter
\def\enumerate{\begingroup\ifnum\@enumdepth>3\@toodeep\else
      \advance\@enumdepth\@ne
      \edef\@enumctr{enum\romannumeral\the\@enumdepth}%
      \topsep\z@\parskip\z@
      \list{\csname label\@enumctr\endcsname}
        {\@nmbrlisttrue\let\@listctr\@enumctr
         \parsep\z@\itemsep\z@\topsep\z@
         \setcounter{\@enumctr}{0}
         \def\makelabel##1{\hss\llap{\rm ##1}}
       }\fi}

\makeatother

\let\bar=\overline
\let\epsilon=\varepsilon
\def\({\big(}
\def\){\big)}

\def\M{\mathfrak M}

\def\0{\underline{0}}

\def\B{\mathscr B}

\def\G{\mathcal G}
\def\H{\mathscr H}

\def\W{ B_n(\delta)}
\def\B{\mathscr B_n}

\DeclareMathOperator{\Char}{char}

\DeclareMathOperator{\Rad}{Rad}

\def\F{\mathcal F}

\def\Std{\mathscr{T}^{std}}

\def\s{\mathfrak s}
\def\ts{\tilde\s}
\def\t{\mathfrak t}

\def\ss{\mathbf s}
\def\ts{\mathbf t}
\def\us{\mathbf u}

\newcommand{\cba}[1]{\mathscr {B}_{#1}}

{\catcode`\|=\active
  \gdef\set#1{\mathinner{\lbrace\,{\mathcode`\|"8000%
                                   \let|\midvert #1}\,\rbrace}}
  \gdef\seT#1{\mathinner{\Big\lbrace\,{\mathcode`\|"8000%
                                   \let|\midverT #1}\,\Big\rbrace}}
}
\def\midvert{\egroup\mid\bgroup}
\def\midverT{\egroup\,\Big|\,\bgroup}

\def\Set[#1]#2|#3|{\Big\{\ #2\ \Big| \
           \vcenter{\hsize #1mm\centering #3}\Big\}}

\title[ Singular parameters for  the Birman-Murakami-Wenzl
algebra]{Singular parameters for the  Birman-Murakami-Wenzl algebra}
\author{Hebing Rui and Mei Si}
\address{H.R. Department of Mathematics,  East China Normal
University, Shanghai, 200241, China} \email{hbrui@math.ecnu.edu.cn}
\address{M.S. Department of Mathematics,  Shanghai Jiaotong
University, Shanghai, 200240, China} \email{simei@sjtu.edu.cn}

\thanks{Both of us  are supported in part by NSFC}

\begin{document}
\baselineskip15pt
\begin{abstract} In this paper, we classify the
singular parameters for the Birman-Murakami-Wenzl algebra over an
arbitrary field. Equivalently, we give a criterion for the
Birman-Murakami-Wenzl algebra being Morita equivalent to the direct
sum of the Hecke algebras associated to certain symmetric groups.
\end{abstract}

\sloppy \maketitle

\section{Introduction}

The Birman-Murakami-Wenzl algebra $\B$ was introduced independently
by Birman, Wenzl~\cite{BW} and Murakami~\cite{Mu} in order to study
the link invariants.  It is cellular over a commutative
ring~\cite{Xi} in the sense of \cite{GL}. Further, Xi classified its
irreducible modules over an arbitrary field~\cite{Xi}.

Recently, Enyang~\cite{E} constructed the Jucys-Murphy basis for
each cell module of $\B$. We lifted Enyang's basis to get the
Jucys-Murphy basis for $\B$~\cite{RS:bmw}. Over certain fields, we
use Jucys-Murphy basis of $\B$ to construct its orthogonal basis.
This enables us to  compute the  Gram determinant associated to each
cell module of  $\B$. Therefore, we  determine explicitly the
semi-simplicity of $\B$ over an arbitrary field~\cite{RS:bmw}. We
remark that Wenzl has got some partial results in \cite{W2}.

 We also
use our results on Gram determinants to classify the blocks of $\B$
over certain fields~\cite{RS2}. Via such results, we determine
explicitly whether the Gram determinant associated to a cell module
is equal to zero or not~\cite{RS2}. This is equivalent to saying
that a cell module of $\B$ is equal to its simple head or not.

Morton and Wassermann~\cite{MW} proved that $\B$ is isomorphic to
the Kauffman tangle algebra~\cite{K}. Further, by specialization,
they proved that the Kauffman tangle algebra is isomorphic to the
Brauer algebra~\cite{B}. Therefore, the Brauer algebra~\cite{B} can
be considered as the classical limit of $\B$.

In \cite{KX}, K$\ddot{\text{o}}$nig and Xi~\cite{KX} proved that the
Brauer algebra can be obtained from some inflations of the group
algebras of certain symmetric groups along certain vector spaces.
They introduced the notion of singular parameters and proved that
the Brauer algebra is Morita equivalent to the direct sum of such
group algebras if the defining parameter is not singular. However,
there is no criterion to determine whether the defining parameter is
singular or not.

The aim of this paper is to give a criterion to determine the
singular parameters for the Brauer algebra.  In fact, we will deal
with $\B$ instead of the Brauer algebra.

We introduce the notion of singular parameters for $\B$ over an
arbitrary field. Via some results on the inflations in \cite{KX}, we
prove that $\B$ is Morita equivalent to the direct sum of  Hecke
algebras associated to certain symmetric groups if the defining
parameters are not singular. Further, we  give an explicit criterion
on the singular parameters for  $\B$ over an arbitrary field. By
specialization, we  also obtain  the explicit criterion on the
singular parameters for Brauer algebras over an arbitrary field.

We organize our paper as follows. In section~2, after recalling the
inflation of an algebra along a vector space in \cite{KX}, we give
Theorem~\ref{main} and Theorem~\ref{main1},  the main results of
this paper, which are the criterions on the singular parameters for
 $\B$ and the Brauer algebra, respectively. In Section~3, we recall some of our
results on the representations of $\B$ over an arbitrary field. We
will use them to prove Theorem~\ref{main} in Section~4.

{\bf Acknowledgement.}  We thank Professor Goodman for explaining
the relationship between the classical limit of the
Birman-Murakami-Wenzl algebra and the Brauer algebra in \cite{MW}.

\section{The main results}
In this section, we state the results on the classification of
singular parameters for $\B$  and Brauer algebras over an arbitrary
field. We start by recalling some results on the inflation of an
algebra along a vector space in \cite{KX}.

Throughout, we assume that  $\kappa$ is a field with characteristic
$\Char(\kappa)$ either zero or $ p$ with $p>0$. By abusing of
notation, we will use $p$ instead of $\Char (\kappa)$. When $\Char
(\kappa)=0$, we set $p=\infty$.

Given  a finite dimensional $\kappa$-space $V$, a $\kappa$-algebra
$B$, and a $\kappa$-bilinear form  $\phi: V\otimes V\rightarrow B$,
K$\ddot{\text{o}}$nig and Xi~\cite{KX} define a $\kappa$-algebra $A$
which is equal to $V\otimes V\otimes B$ as $\kappa$-vector space.
The multiplication of $A$ is defined on basis elements as follows:
\begin{equation}\label{mul} (a\otimes b\otimes x)\cdot (c\otimes
d\otimes y)=a\otimes d\otimes x\phi(b, c)y.\end{equation}
K$\ddot{\text{o}}$nig and Xi~\cite{KX} called this $A$  the
inflation of $B$ along $V$ if there is
 a $\kappa$-linear involution $\sigma$ on $B$ with
$\sigma(\phi(b, c))=\phi(c, b)$ such that  this $\sigma$ can be
extended to  the $\kappa$-linear involution $\tau$ on $A$ satisfying
\begin{equation}\label{mulin}\tau(a\otimes b\otimes x)=b\otimes a\otimes
\sigma(x).\end{equation} In the remaining of this paper, we will use
$\sigma$ instead of $\tau$ if there is no confusion.
 It has been pointed in \cite{KX} that
$A$ may not have a unit.

If $B$ is a  simple $\kappa$-algebra,  there is a unique irreducible
$B$--module which is denoted by L. Let $v_1, v_2, \cdots, v_\ell$ be
a $\kappa$-basis of $V$ with $\dim V=\ell$. K$\ddot{\text{o}}$nig
and Xi~\cite{KX} considered the left $A$-module
\begin{equation}\label{pln}
P(L, i)=V\otimes v_i\otimes L,\end{equation} for some basis element
$v_i$ of $V$ and its $\kappa$-subspace
\begin{equation}\label {nln} N_\phi(L, i)=\{\sum_{v\in V, l\in L}
v\otimes v_i\otimes l \in P(L, \ell)\mid \sum_{v, l} \phi(w, v)l=0,
\forall w\in V\}.\end{equation} They proved that $N_\phi(L, i)$ is
an $A$-submodule of $P(L, i)$ and the corresponding quotient module
$P(L, i)/N_\phi(L, i)$ is irreducible~\cite[Lemma~3.2]{KX}. If
$N_\phi(L, i)\neq 0$, the bilinear form $\phi$ is called
singular~\cite{KX}. The following definition can be found in
\cite{KX}.

\begin{Defn}\label{sing}\cite{KX} Given a $\kappa$-algebra $B$,
 let $\text{Rad} B$ be its  Jacobson radical. Let $\bar
\phi=\pi\circ \phi$ where $\pi: B\rightarrow B/\text{Rad} B$ is the
canonical epimorphism. The bilinear form  $\phi$ is called singular
if $N_{\bar\phi}(L, i)\neq 0$ for some irreducible $B$-module $L$
and some basis element $v_i\in V$.\end{Defn}

The key point is the following theorem, which follows from
Corollary~3.5 and Proposition~4.2 in~\cite{KX}.

\begin{Theorem}\label{main2}\cite{KX}  Given the $\kappa$-algebra $V\otimes V\otimes B$
which is the inflation of the $\kappa$-algebra $B$ along the
$\kappa$-vector space $V$. If $\phi: V\otimes V\rightarrow B$, the
corresponding bilinear form, is non-singular, then $V\otimes
V\otimes B\overset {Morita} \sim B$.\end{Theorem}

We are going to state our main result on $\B$. Throughout, we assume
that $R$ is a commutative ring  which contains the multiplicative
identity $1_R$ and invertible elements $q$, $r$ and $q-q^{-1}$. We
use $\omega$ instead of $q-q^{-1}$ later on.

\begin{Defn}\cite{BW}\cite{Mu} The Birman-Murakami-Wenzl algebra $\B$ with defining parameters
$r$ and $q$  is the unital associative $R$-algebra generated by
$T_i$, $1\le i<n$ subject to the  relations:
\begin{enumerate}
    \item $(T_i-q) (T_i+q^{-1}) (T_i-r^{-1})=0$, for $1\le i<n$,
\item $T_iT_j=T_jT_i$ if $|i-j|>1$,
\item $T_iT_{i+1}T_i=T_{i+1}T_iT_{i+1}$, for $1\le i<n-1$,
\item $E_iT_i=r^{-1} E_i=T_iE_i$, for $1\le i\le n-1$,
\item $E_iT_j^{\pm} E_i=r^{\pm}E_{i}$, for $j=i\pm 1$,
\end{enumerate}
where $E_i=1-\omega^{-1} (T_i-T_i^{-1})$ for $1\le i\le n-1$.
\end{Defn}

It is known that  there is an $R$-linear anti-involution $\ast:
\B\rightarrow \B$ which fixes  $T_i$. If we denote by $\langle
E_1\rangle$ the two-sided ideal of $\B$ generated by $E_1$, then
$\B/\langle E_1\rangle$ is isomorphic to the  Hecke algebra $\H_n$
associated to the symmetric group $\mathfrak S_n$. We denote by
\begin{equation}\label {epsi} \epsilon_n:   \H_n\rightarrow \B/\langle
E_1\rangle\end{equation} the corresponding isomorphism.

 Note that $\H_n$
is the $R$-algebra with generators $g_i$, $1\le i\le n-1$ subject to
the defining relations
$$\begin{cases} (g_i-q)(g_i+q^{-1})=0, & \text{for $1\le i\le n-1$,}\\
g_ig_j=g_jg_i,  & \text{if $|i-j|>1$,}\\
g_ig_jg_i=g_jg_ig_j, & \text{if $|i-j|=1$.}\\
\end{cases}$$
It is known that there is an $R$-linear involution $\ast$ on $\H_n$
which fixes $g_i, 1\le i\le n-1$.

Recall that the symmetric group $\mathfrak S_n$  in $n$ letters is
the Coxeter group with distinguished generators $s_i, 1\le i\le n-1$
subject to the defining relations  $$\begin{cases}  s_i^2=1, & \text{ if $1\le i\le n-1$,}\\
s_is_j=s_js_i, & \text{ if  $|i-j|>1$,}\\ s_is_js_i=s_js_is_j, &
\text{ if $j=i\pm 1$.}\\
\end{cases}$$
 It is known that $s_i$ can be identified with the basic
transposition $(i, i+1)$. For each $w\in \mathfrak S_n$ with reduced
expression $s_{i_1}\cdots s_{i_k}$,  let $T_w:=T_{i_1}T_{i_2}\cdots
T_{i_k}\in \B$. It is well-known that $T_w$ is independent of a
reduced expression of $w$.

Given a non-negative  integer $f\le \lfloor n/2\rfloor$, let $\B^f$
be the two sided ideal of $\B$ generated by $ E^{f, n}$, where
\begin{equation}\label{effe} E^{f, n}=E_{n-1}E_{n-3}\cdots
E_{n-2f+1}.\end{equation} When $f=0$, we denote $E^{0, n}$ by the
identity of $\B$. It is known that
$$\B=\B^0\supset \B^1\supset \cdots \supset \B^{\lfloor n/2\rfloor} \supset 0$$ is a  filtration of two-sided
ideals of $\B$. Let $$s_{i, j}=\begin{cases} s_{i}s_{i+1}\cdots
s_{j-1}, & \text{if  $i<j$,}\\ s_{i-1}s_{i-2}\cdots s_{j},
&\text{if  $i>j$,}\\ 1, &\text{if $i=j$.}\\
\end{cases}$$
Enyang~\cite{Enyang} proved that  $\B^f/\B^{f+1}$ is free over $R$
with basis $S$ where \begin{equation} \label{bff} S =\{T_u^\ast
E^{f, n} T_w T_v \mod \B^{f+1}\mid  u, v\in D_{f, n}, w\in \mathfrak
S_{n-2f}\},
\end{equation}  and
 \begin{equation}\label{dff}D_{f, n}
=\Set[40]s_{n-2f+1,i_f}s_{n-2f+2,j_f}\cdots s_{n-1,i_1}s_{n,j_1}|
      $1\leq i_f<\cdots<i_1\leq n ;\atop1\leq i_k<j_k\leq n-2k+2;1\leq k\leq f$
      |.\end{equation}

Let $\kappa$ be a field which is an $R$-algebra.   Let $$\mathscr
B_{n, \kappa}=\B \otimes_R \kappa.$$ By abusing notation, we will
use $\B$ instead of $\mathscr B_{n, \kappa}$ in the remaining part
of this paper.

 Via the $\kappa$-basis $S$ for $\B^f/\B^{f+1}$ in (\ref{bff}), we
will prove that $\B^f/\B^{f+1}$ is the inflation of $\H_{n-2f}$
along the vector space $V_f$ spanned by $T_u$, $u\in D_{f, n}$.  We
remark that the corresponding anti-involution $\sigma $ on
$\H_{n-2f}$ is the anti-involution $\ast$ on $\H_n$. Further,
$\sigma$ can be extended to $\B^f/\B^{f+1}$, which is the same as
the anti-involution induced by $\ast$ on $\B $.
 The bilinear form \begin{equation}\label{bi} \phi_f: V_f\otimes V_f\rightarrow
 \H_{n-2f}\end{equation}
can be defined via the multiplication of $\B$. Details will be given
in Proposition~\ref{iso}.

The following definition is motivated by K$\ddot{\text{o}}$nig and
Xi's work on Brauer algebras in \cite{KX}.

\begin{Defn}\label{singp} The defining parameters $r$ and $q$ are said to be
singular if there is a positive integer $f\le \lfloor n/2\rfloor$
such that the bilinear form  $\phi_f$ in (\ref{bi}) is singular in
the sense of Definition~\ref{sing}.
\end{Defn}

The following result follows from Theorem~\ref{main2} and
\cite[Lemma~7.1]{KX}, immediately.
\begin{Prop}
Let $\B$ be the Birman-Murakami-Wenzl algebra over $\kappa$. Then
$$\B \overset {Morita}\sim \oplus_{0\le f\le \lfloor n/2\rfloor}
\H_{n-2f}$$ if the defining parameters $r$ and $q$ are not singular.
\end{Prop}

Throughout, we denote by  $e$ the multiplicative order of $q^2$ if
$q^2$ is a root of unity. Otherwise, we define $e=\infty$.

The following result, which is the  main result of this paper, gives
the classification of singular parameters $r$ and $q$ for $\B$ over
an arbitrary field $\kappa$. We define
 \begin{equation}\label{sss}
\mathcal S=\cup_{k=3}^n \{q^{3-2k},\pm q^{3-k},-q^{2k-3},\pm
q^{k-3}\}, \text{for $n\ge 2$.}\end{equation}

\begin{Theorem}\label{main} Let $\B$ be the Birman-Murakami-Wenzl algebra over the field  $\kappa$.
\begin{enumerate}\item Suppose $e>n-2$. \begin{itemize} \item[(1)]
If $r\not\in \{q^{-1}, -q\}$, then   $r$ and $q$ are  singular if
and only if  $r\in  \mathcal S$,
\item [(2)]  If $r\in \{q^{-1}, -q\}$, then $r$ and $q$ are singular if
and only if one of the following conditions holds:
\begin{enumerate}\item $n$ is even or odd with $n\ge 7$, \item
$n=3$, and  $q^4+1= 0$.
\item $n=5$, and $2(q^4+1) (q^6+1)(q^8+1)=0$.\end{enumerate}\end{itemize}
\item If $e\le n-2$, then  $r$ and $q$ are singular if and only
if $r\in \{ q^a, -q^b\mid a, b\in \mathbb Z\}$.
 \end{enumerate}
 \end{Theorem}

Recall that Morton and Wassermann proved that $\B$ is isomorphic the
Kauffman tangle algebra~\cite{MW}.  Morton and Wassermann defined
$\B$ over the commutative ring $\mathbb Z[r^\pm, \omega,
\delta]/\langle  r-r^{-1}-\omega (\delta-1)\rangle $. In their
definition, they do not need the invertibility of $\omega$. By
specializing $\omega$ and $r$ to $0$ and $1$, respectively, they
proved that Kauffman tangle algebra is isomorphic to the Brauer
algebra $B_n(\delta)$~\cite{B}, which is the associative algebra
over $\mathbb Z[\delta]$ generated by $s_i, e_i$, $1\le i\le n-1$
subject to the following relations:
\bigskip
 {\small
\begin{multicols}{2}
\begin{enumerate}
    \item
$s_i^2=1$, for $1\le i<n$.
\item $s_is_j=s_js_i$ if $|i-j|>1$, \item
$s_is_{i+1}s_i=s_{i+1}s_is_{i+1}$,\newline for $1\le i<n-1$,
    \item $e_i^2=\delta e_i$, for $1\le i<n$.
\item $s_ie_j=e_js_i$, if $|i-j|>1$, \item $e_ie_j=e_je_i$, if
$|i-j|>1$.
\item $e_is_i=e_i=s_ie_i$,
\newline for $1\le i\le n-1$, \item
$s_ie_{i+1}e_i=s_{i+1}e_i$,\newline
$e_{i+1}e_is_{i+1}=e_{i+1}s_i$,\newline for $1\le i\le n-2$.
    \item
        $e_{i+1}e_ie_{i+1}=e_{i+1}$ and\newline
        $e_ie_{i+1}e_i=e_i$, for $1\le i\le n-2$.
\end{enumerate}
\end{multicols}}
\bigskip

 Note that the  relationships between Morton--Wassermann's  notations and our
notations are $$\omega=q-q^{-1}, \text{  and  }
\delta=\frac{(q+r)(qr-1)}{r(q+1)(q-1)}.$$
 Therefore, the Brauer
algebra can be obtained from $\B$ by specializing $q, r$ to $1$.

We have $\text{limit}_{q\to 1} \delta\in \mathcal Z$ if $r\in
\mathcal S$  where \begin{equation}\label{zzz} \mathcal Z=\{1, 2,
\cdots, n-2\}\cup\{-2, -4,\cdots, 4-2n\}\cup\{-1,-2, \cdots,
4-n\}.\end{equation} If $r=\pm q^a$, we have $\text{limit}_{q\to 1}
\delta\in \mathbb Z$. Therefore, by Theorem~\ref{main}, we have the
following result immediately.

\begin{Theorem}\label{main1}  Let $\W$ be the Brauer algebra over
the field  $\kappa$. Let $p=\Char \kappa$ if $\Char \kappa>0$ and
let $p=\infty$ otherwise.
\begin{enumerate} \item Suppose $p>n-2$. \begin{enumerate} \item  If
$\delta\neq 0$, then  $\delta$ is singular if and only if
$\delta=a\cdot 1_\kappa$ and $a\in \mathcal Z$,
\item If $\delta=0$, then $\delta$ is singular if and only if one of the following conditions holds:
\begin{itemize}\item[(1)] $n$ is either even or odd with $n>7$,
\item [(2)] $n=3$ and $p= 2$. \end{itemize}
\end{enumerate}
\item Suppose $p\le n-2$. Then $\delta $ is singular if and only if
$\delta =a \cdot 1_\kappa$ and  $a\in \mathbb Z$.
\end{enumerate}
\end{Theorem}

By comparing Theorems~\ref{main}-\ref{main1}, we find that $0$ is
not singular for $B_5(0)$ if $p>3$.  The reason is that
$\text{limit}_{q\to 1} 2(q^4+1)(q^6+1)(q^8+1)\neq 0 $ in $\kappa$ if
$p>3$.

 K$\ddot{\text{o}}$nig and Xi  pointed out  that the singular
parameter $\delta $ is dependent of $\delta$ only~\cite[p1502]{KX}.
This is not the same as  our Theorem~\ref{main1}. One can find the
counter-example by considering the cell module $\Delta(1, (1))$ when
$p=2$ and  $\delta=0$.

Finally, we remark that  Theorem~\ref{main1} can be proved by
similar arguments for $\B$. In order to give the detailed proof, we
need results which are similar to Theorem~\ref{simple},
Definition~\ref{adm}, Theorem~\ref{BMW1} and Lemma~\ref{functor}
etc, which can found in \cite{CMW, GL, Na,  RS3, RS4} etc. We leave
the details to the reader.

\section{Representations of Birman-Murakami-Wenzl algebras}
In this section, we recall some results on the representations of
$\B$ over a field . We will use them to prove Theorem~\ref{main} in
section~4. We start by recalling some combinatorics.

Recall that a \textsf{partition} of $n$ is a weakly decreasing
sequence  of non--negative integers
$\lambda=(\lambda_1,\lambda_2,\dots)$ such that
$|\lambda|:=\lambda_1+\lambda_2+\cdots=n$. In this case, we write
$\lambda\vdash n$. Let $\Lambda^+(n)=\{\lambda\mid \lambda\vdash
n\}$. Then $\Lambda^+(n)$ is  a poset with dominance order
$\trianglelefteq$ as the partial order on it. More explicitly,
$\lambda\trianglelefteq \mu$ for $\lambda, \mu\in \Lambda^+(n)$  if
$ \sum_{j=1}^i \lambda_j\le \sum_{j=1}^i \mu_j$ for all possible
$i$. Write $\lambda\vartriangleleft \mu$ if $\lambda\trianglelefteq
\mu$ and $\lambda\ne \mu$.

Suppose that  $\lambda$ and $\mu$ are two partitions. We say that
$\mu$ is obtained from $\lambda$ by \textsf{adding} a box if there
exists an $i$  such that $\mu_i=\lambda_i+1$ and $\mu_j=\lambda_j$
for $j\neq i$. In this situation we will also say that $\lambda$ is
obtained from $\mu$ by \textsf{removing} a box and we write
$\lambda\rightarrow\mu$ and $\mu\setminus\lambda=(i,\lambda_i+1)$.
We will  say that the pair  $(i,\lambda_i+1)$ is an \textsf{addable}
node of $\lambda$ and a \textsf{removable} node of $\mu$. Note that
$|\mu|=|\lambda|+1$.

The Young diagram $[\lambda]$ for a partition $\lambda=(\lambda_1,
\lambda_2, \cdots)$ is a collection of boxes arranged in
left-justified rows with $\lambda_i$ boxes in the $i$-th row of
$[\lambda]$. A $\lambda$-tableau $\ss$ is obtained by inserting $i,
1\le i\le n$ into $[\lambda]$ without repetition. The symmetric
group $\mathfrak S_n$ acts on $\ss$ by permuting its entries. Let
$\ts^\lambda$ be the $\lambda$-tableau obtained from the Young
diagram $Y(\lambda)$ by adding $1, 2, \cdots, n$ from left to right
along the rows. For example, for $\lambda=(4,3,1)$,
$$\ts^{\lambda}=\young(1234,567,8).$$

If $\ts^\lambda w=\ss$, write $w=d(\ss)$. Note that $d(\ss)$ is
uniquely determined by $\ss$.

A $\lambda$-tableau $\ss$ is standard if the entries in $\ss$ are
increasing both from left to right in each row and from top to
bottom in each column. Let $\Std_n(\lambda)$ be the set of all
standard $\lambda$-tableaux.

For $\lambda\vdash n-2f$, let $\mathfrak S_{\lambda}$ be the Young
subgroup of $\mathfrak S_{n-2f}$ generated by $s_j$, $1\le j\le
n-2f-1$ and $j\neq \sum_{k=1}^i \lambda_k$ for all possible $i$.

Let $\Lambda_n=\set{(f, \lambda)\mid \lambda\vdash n-2f, 0\le f\le
\lfloor\frac n 2\rfloor}$. Given  $(k, \lambda),  (f, \mu)\in
\Lambda_n$, define  $(k, \lambda)\trianglelefteq (f, \mu)$ if either
$k<f$ or $k=f$ and $\lambda\trianglelefteq\mu$. Write $(k,
\lambda)\lhd(f, \mu)$, if  $(k, \lambda)\unlhd (f, \mu)$ and $(k,
\lambda)\ne (f, \mu)$.

 Let $I(f, \lambda)=\Std_n(\lambda)\times D_{f, n}$ where
 $D_{f, n}$ is defined in (\ref{dff}). Define
\begin{equation}\label{basise}
 C_{(\ss, u) (\ts, v)}^{(f, \lambda)} =T_u^\ast T_{d(\ss)}^{\ast}
  \mathfrak M_\lambda
T_{d(\ts)}T_v,\quad (\ss, u), (\ts, v)\in I(f, \lambda)
\end{equation}
where $\mathfrak M_\lambda=E^{f, n} X_{\lambda}$, $E^{f, n}=E_{n-1}
E_{n-3}\cdots E_{n-2f+1}$, $X_\lambda=\sum_{w\in \mathfrak
S_\lambda} q^{l(w)}T_ w$, and $l(w)$,  the length of $w\in \mathfrak
S_n$.

\begin{Theorem}\cite{Enyang} \label{cell} Let $\B$ be the Birman-Murakami-Wenzl
algebra over  $R$. Let $\ast : \B\rightarrow \B$ be the $R$-linear
anti-involution which fixes $T_i, 1\le i\le n-1$. Then
\begin{enumerate}
\item $\mathscr C_n =\left\{ C_{(\ss, u) (\ts, v)}^{(f, \lambda)} \mid (\ss, u), (\ts, v)\in
I(f, \lambda), \lambda\vdash n-2f,  0\le f\le \lfloor\frac
n2\rfloor\right\}$ is a free $R$--basis of $\B$.
\item $\ast (C_{(\ss, u) (\ts, v)}^{(f, \lambda)})=C_{(\ts, v) (\ss,
u)}^{(f, \lambda)}$.
\item For any $h\in \B$,
$$h \cdot C_{(\ss, u) (\ts, v)}^{(f, \lambda)}  \equiv \sum_{(\us, w)\in I(f, \lambda)} a_{\us,
w} C_{(\us, w) (\ts, v)}^{(f, \lambda)} \mod \B^{\vartriangleright
(f, \lambda)}$$
 where
$\B^{\vartriangleright (f, \lambda)}$ is the free $R$-submodule
generated by $ C_{(\tilde \s, \tilde u) (\tilde \ts, \tilde v)}^{(k,
\mu)}$ with $(k, \mu) \vartriangleright (f, \lambda)$ and $(\tilde
\ss, \tilde u), (\tilde \ts, \tilde v)\in I(k, \mu)$.
 Moreover, each coefficient $a_{\us, w}$ is independent of $(\ts,
v)$.
\end{enumerate}
\end{Theorem}

Theorem~\ref{cell} shows that $\mathscr C_n $ is a cellular basis of
$\B$ in the sense of \cite{GL}. We remark that Xi~\cite{Xi} first
proved that $\B$ is  cellular in the sense of \cite{GL}. The
cellular basis $\mathscr C_n$ was constructed by Enyang in
\cite{Enyang}.

 In this paper, we will only consider
left modules for $\B$. By  general theory about cellular algebras in
\cite{GL}, we know that,  for each $(f, \lambda)\in \Lambda_n$,
there is a cell module $\Delta(f, \lambda)$ of $\cba{n}$, spanned by
$$\set{  T^\ast _{v} T^\ast_{d(\ts)}\M_\lambda\mod \cba{n}^{\vartriangleright
(f, \lambda)}\mid (\ts, v)\in I(f, \lambda)}.$$ Further, there is an
invariant  form $\phi_{f,\lambda}$ on  $\Delta(f, \lambda)$. Let
$$\Rad\Delta(f, \lambda)
   =\set{x\in\Delta(f, \lambda)|\phi_{f, \lambda}(x,y)=0\text{ for all }
                          y\in\Delta(f, \lambda)}.$$ Then
                          $\Rad\Delta(f, \lambda)$
is a $\B$--submodule of $\Delta(f, \lambda)$. Let $$D^{f,
\lambda}=\Delta(f, \lambda)/\Rad\Delta(f, \lambda).$$

Recall that  $e$  is  the order of $q^2$ if $q^2$ is a root of
unity. Otherwise,  $e=\infty$. We say that the partition $\lambda$
is \textsf{ $e$-restricted} if $\lambda_i-\lambda_{i+1}<e$ for all
possible $i$.

\begin{Theorem}\cite{Xi}\label{simple} Let $\B$ be the Birman-Murakami-Wenzl algebra
over the field $\kappa$ which contains non-zero  parameters $r, q,$
and $q-q^{-1}$.

\begin{enumerate}
\item Suppose $ r \not\in \{q^{-1}, -q\}$. The non-isomorphic irreducible
$\B$-modules are indexed by $(f, \lambda)$ and $\lambda$ is
$e$-restricted.  \item Suppose $ r \in \{q^{-1},
-q\}$.\begin{enumerate}
\item If $n$ is odd, then the non-isomorphic irreducible
$\B$-modules are indexed by $(f, \lambda)$, $0\le f\le \lfloor
n/2\rfloor$. $\lambda\in \Lambda^+(n-2f)$ and $\lambda$ is
$e$-restricted.
\item If $n$ is even, then the non-isomorphic irreducible
$\B$-modules are indexed by $(f, \lambda)$, $0\le f<\lfloor
n/2\rfloor$. $\lambda\in \Lambda^+(n-2f)$,  $\lambda$ is
$e$-restricted.
\end{enumerate}
\end{enumerate}
\end{Theorem}

By general results on cellular algebras in \cite{GL},  $\B$ is
(split) semisimple over $\kappa$ if and only if $D^{f,
\lambda}=\Delta(f, \lambda)$ for all $(f, \lambda)\in \Lambda_n$. We
remark that the authors have given the necessary and sufficient
conditions for $\B$ being seimisimple over an arbitrary
field~\cite{RS:bmw}. However, we will not need this result in this
paper. What we need is the explicit criterion for $\Delta(f,
\lambda)$ being equal to its simple head $D^{f, \lambda}$. In other
words, we give a criterion to
  determine when the Gram determinant $\det G_{f,\lambda}\neq 0$.
   Here $G_{f, \lambda}$ is the Gram matrix associated to the
invariant form $\phi_{f, \lambda}$. We need some combinatorics to
state this result.

For each box $p=(i, j)\in [\lambda]$, define $p^+=(i, j+1)$,
$p^-=(i+1, j)$, and
\begin{equation}\label{lamcont}c_\lambda(p)= r
 q^{2(j-i)}.\end{equation} We will use $c(p)$ instead of
 $c_\lambda(p)$.

 Given two partitions $\lambda$ and $\mu$, we write $\lambda\supset
 \mu$ if $\lambda_i\ge \mu_i$ for all possible $i$. In this case,
 the corresponding skew Young diagram $[\lambda/\mu]$ can be
 obtained from $[\lambda]$ by removing all nodes in $[\mu]$.

We recall the definition of  $(f, \mu)$-admissible partition
$\lambda$ in \cite{RS2}. In the Definition~\ref{adm},  we assume
that $q^2 $ is not a root of unity and the ground field $\kappa$ is
the complex field $\mathbb C$ although we give this definition  with
loose restrictions in \cite{RS2}.

\begin{Defn}\label{adm}\cite{RS2}  Given $\lambda\in \Lambda^+(n)$ and $\mu\in
\Lambda^+(n-2f)$, we say  $\lambda$ is  $(f, \mu)$-admissible over
the  field $\kappa$ if \begin{enumerate}\item $\lambda\supset
\mu$,\item
 there is a pairing of nodes
$p_i, \tilde p_i$, $1\le i\le f$ in $[\lambda/\mu]$ such that
$c(p_i)c(\tilde p_i)=1$.   We call $\{p_i, \tilde p_i\}$ an
admissible pair. Further, there are two possible configurations of
nodes in $[\lambda/\mu]$ as follows.

\begin{figure}[ht]
\unitlength 1mm 
\linethickness{0.4pt}
\ifx\plotpoint\undefined\newsavebox{\plotpoint}\fi 
\begin{picture}(120,29.5)(10,70)
\put(28.25,94.75){\framebox(11.5,4.75)[cc]{}}
\put(34,90.25){\framebox(10.75,4.5)[cc]{}}
\put(40,85){\framebox(11.25,5.25)[cc]{}}
\multiput(42.93,84.68)(.07108,-.03186){12}{\line(1,0){.07108}}
\multiput(44.64,83.92)(.07108,-.03186){12}{\line(1,0){.07108}}
\multiput(46.34,83.15)(.07108,-.03186){12}{\line(1,0){.07108}}
\multiput(48.05,82.39)(.07108,-.03186){12}{\line(1,0){.07108}}
\multiput(49.75,81.62)(.07108,-.03186){12}{\line(1,0){.07108}}
\multiput(51.46,80.86)(.07108,-.03186){12}{\line(1,0){.07108}}
\multiput(53.17,80.09)(.07108,-.03186){12}{\line(1,0){.07108}}
\multiput(54.87,79.33)(.07108,-.03186){12}{\line(1,0){.07108}}
\multiput(56.58,78.56)(.07108,-.03186){12}{\line(1,0){.07108}}
\multiput(49.43,84.68)(.07576,-.03283){11}{\line(1,0){.07576}}
\multiput(51.1,83.96)(.07576,-.03283){11}{\line(1,0){.07576}}
\multiput(52.76,83.24)(.07576,-.03283){11}{\line(1,0){.07576}}
\multiput(54.43,82.51)(.07576,-.03283){11}{\line(1,0){.07576}}
\multiput(56.1,81.79)(.07576,-.03283){11}{\line(1,0){.07576}}
\multiput(57.76,81.07)(.07576,-.03283){11}{\line(1,0){.07576}}
\multiput(59.43,80.35)(.07576,-.03283){11}{\line(1,0){.07576}}
\multiput(61.1,79.62)(.07576,-.03283){11}{\line(1,0){.07576}}
\multiput(62.76,78.9)(.07576,-.03283){11}{\line(1,0){.07576}}
\put(44.75,90){\line(0,-1){4.75}}
\put(39.75,94.25){\line(0,-1){4.25}}
\put(54.75,72.75){\framebox(11.5,5.25)[cc]{}}
\put(60.5,77.5){\line(0,-1){5}}
\put(102.5,89.5){\line(0,1){0}}
\multiput(107.18,87.18)(.048913,-.032609){16}{\line(1,0){.048913}}
\multiput(108.74,86.14)(.048913,-.032609){16}{\line(1,0){.048913}}
\multiput(110.31,85.09)(.048913,-.032609){16}{\line(1,0){.048913}}
\multiput(111.88,84.05)(.048913,-.032609){16}{\line(1,0){.048913}}
\multiput(113.44,83.01)(.048913,-.032609){16}{\line(1,0){.048913}}
\multiput(115.01,81.96)(.048913,-.032609){16}{\line(1,0){.048913}}
\multiput(116.57,80.92)(.048913,-.032609){16}{\line(1,0){.048913}}
\multiput(118.14,79.88)(.048913,-.032609){16}{\line(1,0){.048913}}
\multiput(119.7,78.83)(.048913,-.032609){16}{\line(1,0){.048913}}
\multiput(121.27,77.79)(.048913,-.032609){16}{\line(1,0){.048913}}
\multiput(122.83,76.74)(.048913,-.032609){16}{\line(1,0){.048913}}
\multiput(124.4,75.7)(.048913,-.032609){16}{\line(1,0){.048913}}
\multiput(107.18,81.18)(.046218,-.032213){17}{\line(1,0){.046218}}
\multiput(108.75,80.08)(.046218,-.032213){17}{\line(1,0){.046218}}
\multiput(110.32,78.99)(.046218,-.032213){17}{\line(1,0){.046218}}
\multiput(111.89,77.89)(.046218,-.032213){17}{\line(1,0){.046218}}
\multiput(113.47,76.8)(.046218,-.032213){17}{\line(1,0){.046218}}
\multiput(115.04,75.7)(.046218,-.032213){17}{\line(1,0){.046218}}
\multiput(116.61,74.61)(.046218,-.032213){17}{\line(1,0){.046218}}
\multiput(118.18,73.51)(.046218,-.032213){17}{\line(1,0){.046218}}
\multiput(119.75,72.42)(.046218,-.032213){17}{\line(1,0){.046218}}
\multiput(121.32,71.32)(.046218,-.032213){17}{\line(1,0){.046218}}
\multiput(122.89,70.23)(.046218,-.032213){17}{\line(1,0){.046218}}
\put(125.25,67.25){\framebox(4.75,9.75)[cc]{}}
\put(125.5,72.5){\line(1,0){4.5}}
\put(34,99.5){\line(0,-1){4.25}}
\put(38.25,71.5){(a)} \put(107.25,70.25){(b)}
\put(93.5,89.5){\framebox(4.5,8.75)[cc]{}}
\put(93.75,94){\line(1,0){5.25}}
\put(98,84.25){\framebox(4,9.75)[cc]{}}
\put(102,79.5){\framebox(4.25,10)[cc]{}}
\put(97.75,89.5){\line(1,0){4.5}}
\put(102.25,84.5){\line(0,-1){.25}}
\put(102.5,84){\line(1,0){3.5}}
\end{picture}
\caption{}
\end{figure}

\item  the number of
columns in Figure~1(b) is even if  $c(p)= q$ and $\{p, p^-\}$ is an
admissible pair which is contained in Figure~1(b).
\item the number of rows in  Figure~1(a) is even if  $c(p)=-q^{-1}$ and
 $\{p, p^+\}$ is an admissible pair which is contained in Figure~1(a).
 \end{enumerate}\end{Defn}

Given a  $\lambda\in \Lambda^+(n)$ and a node $(i,j)\in [\lambda]$,
let
$$h_{ij}^\lambda=\lambda_i-j+\lambda_j'-i+1,$$ where $\lambda'$ is
the dual partition of $\lambda$.  This $h_{ij}^\lambda$ is known as
the $(i,j)$-hook length in $[\lambda]$.

Recall that $e$ is the order of  $q^2$ if it is a  root of unity.
Otherwise,  $e=\infty$. Let $p=\Char(\kappa)$ if $\Char(\kappa)>0$
and  let $p=\infty$ if
 $\Char(\kappa) =0$.
For each integer $h$, define
$$ \nu_{e,p}(h)=\begin{cases}\nu_p(\frac{h}{e}),& \text{if}~ e<\infty ~\text{and}~  e|h;\\
-1,& otherwise.\end{cases}$$ where $\nu_p(h)$ is the largest power
of $p$ dividing $h$ if $p$ is finite and $\nu_{\infty}(h)=0$ if
$p=\infty$,

In the following result, $\kappa$ is an arbitrary field.

\begin{Theorem}\label{BMW1}\cite{RS2} Suppose $\kappa$ is a field  which contains invertible $q, r$ and $(q-q^{-1})^{-1}$.
For each $(f, \lambda)\in \Lambda_{n}$, let $\det G_{f, \lambda}$ be
the Gram determinant associated to the cell module $\Delta(f,
\lambda)$. Then $\det G_{f,\lambda}\neq 0$ if and only if the
following conditions hold:
\begin{enumerate}\item  $r\neq \pm q^a$ where the integer  $a$ and the sign of $q^a$ are determined by
 $(f-\ell, \lambda)$-admissible partitions  over  $\mathbb C $
 with $\mathbf q\in \mathbb C$, $o(\mathbf q^2)=\infty$ and  $0\le \ell\le f-1$,
 \item $\lambda$ is $e$-restricted,\item
$\nu_{e,p}(h_{ac}^\lambda)=\nu_{e,p}(h_{ab}^\lambda), \forall (a,c),
(a,b)\in [\lambda]$.\end{enumerate}
\end{Theorem}

We are going to prove that $\B$ can be obtained from $\H_{n-2f}$,
$0\le f\le \lfloor n/2\rfloor$ by inflations along certain vector
spaces $V_f$. When $f=\lfloor n/2\rfloor$, we denote by $\H_{n-2f}$
the ground field $\kappa$.

It is proved in
 \cite{BW} that \begin{equation}\label{e1e} E_{n-1} \mathscr{B}_n E_{n-1}=E_{n-1} \mathscr{B}_{
 n-2}.\end{equation}
Therefore, applying (\ref{e1e}) repeatedly yields
\begin{equation}\label{efe} E^{f, n} \mathscr{B}_n E^{f, n}=E^{f, n} \mathscr{B}_{n-2}\end{equation} for all
positive integers $f\le \lfloor n/2\rfloor$,

Now, let $V_f$ be the free $\kappa$-module generated by $T_u$ with
$u\in D_{f, n}$. By (\ref{efe}), for any $u,v\in V_f$, we have
$E^{f,n} T_u T_v^\ast E^{f, n}= E^{f, n} h$ for some $h\in \mathscr
B_{ n-2f}$. Note that $E^{f, n} h=0$ for $h\in \mathscr B_{ n-2f}$
if and only if $h=0$. This gives rise to  a well-defined bilinear
form
\begin{equation} \label{phi} \phi_f: V_f \otimes V_f \rightarrow
\H_{ n-2f}\end{equation} such that $\phi_f(T_u,
T_v)=\epsilon_{n-2f}^{-1} (h)$ and $\epsilon_{n-2f}$ is given in
(\ref{epsi}).
\begin{Prop} \label{iso} Let $\B$ be the Birman-Murakami-Wenzl algebra over
a field  $\kappa$.  For any non-negative integer $f\le \lfloor
n/2\rfloor$, we have  $\B^f/\B^{f+1}\cong V_f\otimes V_f\otimes \H_{
n-2f}$ as $\kappa$-algebras.
\end{Prop}
\begin{proof} For any $h\in \mathscr B_{n-2f}$, let  $h'$ be the image of
$h$ in $\mathscr B_{n-2f}/\langle E_1\rangle$. Then $T_w'=g_w$ for
all $w\in \mathfrak S_{n-2f}$.  For all $u, v\in D_{f, n}$ ,  by
(\ref{bff}), the $\kappa$-linear map sending $T^\ast_v E^{f, n} T_w
T_u\text{ mod } \B^{f+1}$ to $T_v\otimes T_u\otimes g_w$ is the
required isomorphism. The required anti-involution on $\H_{n-2f}$ is
$\ast$ and the required anti-involution on  $V_f\otimes V_f\otimes
\H_{n-2f}$ can be defined as that for inflation of an algebra along
a vector space in (\ref{mulin}).
\end{proof}

We are going to recall some results on the representations of the
Hecke algebra $\H_{n}$ over an arbitrary field $\kappa$.

Via Kazhdan-Lusztig basis for $\H_n$, Graham and Lehrer have proved
 that $\H_n$ is cellular over $\kappa$ in the sense of
\cite{GL}. In this case, the corresponding poset is $\Lambda^+(n)$.
We will use the cell module defined via Murphy basis
$\{x_{\ts\ss}\mid \ss,\ts\in \Std_n(\lambda)\}$ for $\H_n$. We do
not need the explicit construction of $x_{\ts\ss}$. In fact,  what
we need is the fact that the cell module $\Delta(0, \lambda)$ for
$\B$ with $\lambda\in \Lambda^+(n)$ can be considered as the cell
module for $\H_n$ defined via $\{x_{\ts\ss}\mid \ss,\ts\in
\Std_n(\lambda)\}$. We denote the corresponding cell module by
$S^\lambda$. Its $\kappa$-basis elements are denoted by $\{x_\ss\mid
\ss\in \Std_n(\lambda)\}$. We remark that   $S^\lambda$ is known as
dual Specht module for $\H_n$.

Let  $\phi_{ \lambda}$ be the invariant form on $S^\lambda$ and let
$\Rad S^\lambda$ be the radical of $\phi_\lambda$. Then $\Rad
S^\lambda$ is an $\H_n$--submodule of $S^\lambda$. It follows from
general results on cellular algebras in \cite{GL} that  $D^{
\lambda}=S^\lambda/\Rad S^\lambda$ is either zero or absolutely
irreducible. It is known that  $D^{ \lambda}\neq 0$ if and only if
$\lambda$ is $e$-restricted.

 By (\ref{pln}) and (\ref{nln}), the $\B^f/\B^{f+1}$-modules $P(D^\lambda, \ell)$ and  $N_{\phi_f}(D^\lambda, \ell)$
are well defined where $D^\lambda$ is the simple head of
$S^\lambda$, $1\le \ell\le \dim V_f$, and $V_f, \phi_f$ are given
in (\ref{phi}).

The following result which can be verified directly, sets up the
relationship between $P(D^\lambda, \ell)$, $\Delta(f, \lambda)$ etc.
Note that for each $(f, \lambda)\in \Lambda_n$, $\B^{f+1}$ acts on
$\Delta(f, \lambda)$, trivially. Therefore, it can be   considered
as $\B^f/\B^{f+1}$-module.

\begin{Prop} \label{tool} Suppose  $0<f\le \lfloor n/2\rfloor$ and $\lambda\in \Lambda^+(n-2f)$.
 As $\B^f/\B^{f+1}$-modules, we have
\begin{enumerate}
\item $V_f\otimes v_\ell\otimes  S^\lambda\cong \Delta(f, \lambda)$
for each basis element $v_\ell\in V_f$.
\item  $V_f \otimes v_\ell\otimes
Rad S^\lambda$ is a submodule of $V_f\otimes v_\ell\otimes
S^\lambda, \forall$ $e$-restricted partitions  $\lambda$. The
corresponding quotient module is isomorphic to $P(D^\lambda, \ell)$.
\item $P(D^\lambda, \ell)/N_{\phi_f} (D^\lambda, \ell)\cong D^{f,
\lambda}$, for all $e$-restricted partitions  $\lambda$. \item If
$D^\lambda=S^\lambda$, then $N_{\phi_f}(D^\lambda,\ell)\cong \Rad
\Delta(f,\lambda)$.
\end{enumerate}
\end{Prop}

\section{Proof of Theorem~\ref{main}}
In this section, we always assume that $r, q$ are defining
parameters for $\B$. Recall that $e$ is the order of the $q^2$ if
$q^2$ is a root of unity. Otherwise, $e=\infty$.

\begin{Prop}\label{B1}
Suppose $e>n-2$.
\begin{itemize}\item[(1)] If $r\not\in\{q^{-1},-q\}$, then $r$ and $q$ are  singular
 if and only if $r\in \mathcal S$  given in (\ref{sss}).
\item[(2)] Suppose $r\in\{q^{-1},-q\}$. Then $r$ and $q$ are singular if and only if one of the following
conditions holds \begin{enumerate} \item  $n$ is either even or odd
with $n\ge 7$,
\item  $n=3$ and  $q^4+1=0$.
\item  $n=5$ and  $2(q^4+1)(q^6+1) (q^8+1)=0$.
\end{enumerate}\end{itemize}
\end{Prop}

\begin{proof}  Since we are assuming $e>n-2$,  $\H_{n-2f}$ is semisimple over $\kappa$  for any  positive integer $f\le [\frac{n}{2}]$.
In particular,   $S^\lambda=D^\lambda$, $\forall \lambda\in
\Lambda^+(n-2f)$. By Proposition~\ref{tool}(d), $
N_{\phi_f}(D^\lambda,i)\cong \Rad \Delta(f,\lambda)$ for some $i,
1\le i\le \dim V_f$. Therefore, $r, q$ are singular if and only if
$\det G_{f, \lambda}=0$ for some $(f, \lambda)\in \Lambda_n$ with
$0<f\le \lfloor n/2\rfloor$.

By  \cite[Prop.~5.1, Coro. 4.25]{RS:bmw}, we have proved that $r\in
S$ and $r\not\in\{q^{-1}, -q\}$ if and only if   $\det G_{1,
(k-2)}\det G_{1, (1^{k-2})}=0$ for some integer $k$ with $2\le k\le
n$. Further, in the proof of \cite[Prop. 5.6]{RS:bmw}, we have
proved that there is an $(f, \lambda)\in \Lambda_n$ with $f>0$  such
that
 $\det G_{f, \lambda}=0$ provided  $\det G_{1,
(k-2)}\det G_{1, (1^{k-2})}=0$ for some integer $k$ with $2\le k\le
n$. Therefore, $r$ and $q$ are singular. This proves (1).

Suppose $r\in \{-q, q^{-1}\}$. In \cite[p177]{RS:bmw}, we have given
the explicit formulae on $\det G_{1, \lambda}$ for $\lambda\in \{(1)
, (3), (1^3), (2, 1)\}$, and $\det G_{2, (1)}$ up to some invertible
elements in $\kappa$. We list these formulae as follows. One can use
the Gap program  \cite{Gap} to verify them easily.
\begin{itemize} \item [(1)] $
\det G_{1, (1)}=(q^4+1)$ if  $r\in \{q^{-1}, -q\}$.\item [(2)] $\det
G_{1, (3)}= 2^5[2]^{10}[3]^{14}(1+q^8)$ if  $ r = -q$,
\item [(3)] $\det G_{1, (3)}=-[2]^{10}[3]^{11}(1+q^4)^6$ if
$ r = q^{-1}$,  \item [(4)]  $\det G_{1, (1,1,1)}=[3](1+q^4)^6$  if
$ r = -q$,
\item [(5)] $\det G_{1,
(1,1,1)}= 2^5 [3]^4 (1+q^8)$ if  $ r = q^{-1}$,\item [(6)] $\det
G_{1,(2, 1)}=-[2]^4[3]^{15}
 (1+q^6)^4$ if $r\in \{q^{-1}, -q\}$.
\item [(7)]
$\det G_{2, (1)}=-2^6(1+q^2) (1+q^4)^{10}(1+q^6)$ if $r\in \{q^{-1},
-q\}$,
\end{itemize}
where  $[a]=\frac{q^a-q^{-a}}{q-q^{-1}}$ for  $a\in \mathbb Z$.

 Since we
are assuming that $e>n-2$, we have $e>3$ if $n=5$ and $e\ge 2$ if
$n=3$. In each case, $S^\lambda=D^\lambda$. Therefore, (2)(b)-(c)
follows from these explicit formulae, immediately.  In
\cite[P177]{RS:bmw}, we have also proved that
 $\det G_{n/2, 0}=0$ for even  $n$. This proves the first part of (2)(a).

Suppose that  $n$ is odd, with $n\ge 7$. We have  $e>5$ and hence
$S^\lambda=D^\lambda$ for
  $\lambda=(3, 2)$. At the end of the proof of \cite[Prop. 5.8]{RS:bmw}, we have
 proved   $\det G_{\frac{n-5}2, (3, 2)}=0$, which  forces $r$ and $q$
 being singular.
 \end{proof}

 We are going to deal with the case $e\le n-2$. The following result may be well known. We
include a proof here.

\begin{Lemma}\label{hom} Suppose that $A_n$ are $\kappa$-algebras for all positive integers $n$ such that   $A_{n-1}$ is a subalgebra of $A_{n}$.
If $A_n$-modules $M$ and $N$ have  filtrations  of $A_{n-1}$-modules
$$\begin{aligned} & 0=M_0\subseteq M_1\subseteq M_2\subseteq  \cdots  \subseteq M_{\ell-1}\subseteq M_\ell=M\\
& 0=N_0\subseteq N_1\subseteq N_2\subseteq  \cdots  \subseteq N_{k-1}\subseteq N_k=N.\\
\end{aligned}$$
and if  $\text{Hom}_{A_n}(M,N)\neq 0$,  then there exist some
integers $i$ and $j$,  $0\le i\le \ell-1$ and $0\le j\le k-1$ such
that $\text{Hom}_{A_{n-1}}(M_{i+1}/M_i, N_{j+1}/N_j)\neq 0$.
\end{Lemma}

\begin{proof} Acting the left exact functor $\text{Hom}_{A_{n-1}}(M,-)$
on the short exact sequence
 $$0\rightarrow N_{i-1}\rightarrow N_i\rightarrow N_i/N_{i-1}\rightarrow 0$$
 and using induction on $i$,  we can find   some integer $j$, $0\le j\le \ell-1$ such that $\text{Hom}_{A_{n-1}}(M, N_{j+1}/N_j)\neq 0$.  Acting the left exact functor
 $\text{Hom}_{ A_{n-1}}(-,N_{j+1}/N_{j})$
on the short exact sequence
 $$0\rightarrow M_{i-1}\rightarrow M_i\rightarrow M_i/M_{i-1}\rightarrow 0$$
and using induction on $i$,  we can find an $i, 0\le i\le k-1$ such that
 $\text{Hom}_{ A_{n-1}}(M_{i+1}/M_i, N_{j+1}/N_j)\neq 0$, as required.
\end{proof}

Let $\B$-mod be the category of left $\mathscr B_n$-modules, by (\ref{e1e}), $E_{n-1} \mathscr B_n E_{n-1}=E_{n-1}\mathscr B_{n-2}$. Therefore, one can  define  two functors
$$\mathcal F_n: \mathscr B_n\text{-mod}\rightarrow \mathscr B_{n-2}\text{-mod},
\text{ and } \G_{n-2}: \mathscr B_{n-2}\text{-mod} \rightarrow \mathscr B_n\text{-mod}$$ such that
$$ \mathcal F_n(M) =E_{n-1} M   \text{ and } \mathcal G_{n-2}(N)=\mathscr B_n  E_{n-1}\otimes_{\mathscr B_{n-2}}  N   ,$$ for all left $\mathscr B_n$-modules $M$ and left
$\mathscr B_{n-2}$-modules $N$. For the Brauer algebras, such two
functors have been studied by Doran, Wales and Hanlon in~ \cite{DWH}.

For the simplification of notation, we will use $\mathcal F,
\mathcal G$ instead of $\mathcal F_n$ and $\mathcal G_{n-2}$,
respectively. Note that $\mathcal F$ is an exact functor and
$\mathcal G$ is a right exact functor. We will denote
$\B^f/\B^{f+1}$-module $P(D^\lambda, \ell)$ by $P_n(D^\lambda,
\ell)$

\begin{Lemma} \label{functor}
 Suppose that $(f, \lambda)\in \Lambda_{n}$.
\begin{enumerate}
\item $\mathcal F\mathcal G=1$.
\item $\mathcal G(\Delta(f, \lambda) )=\Delta(f+1,
\lambda)$.
\item $\mathcal F(\Delta(f, \lambda ))=\Delta(f-1, \lambda)$.
\item Let $\lambda$ be $e$-restricted. Then $G(P_n(D^\lambda, \ell))=
P_{n+2}( D^\lambda, \ell)$ for any integer $\ell, 1\le \ell\le \dim
V_f$.
\item  Let $\lambda$ be $e$-restricted.  Then $F(P_n(D^\lambda, \ell))=
P_{n-2}( D^\lambda, \ell)$ for any integer $\ell, 1\le \ell\le \dim
V_f$.
\item For each $\B$-module $N$, if $\phi: \Delta(f,
\lambda)\rightarrow N$ is non-zero, and if $f\ge 1$, then $\mathcal
F(\phi)\neq 0$.\end{enumerate}
\end{Lemma}

\begin{proof}We have proved (a)-(c) and (f)  in \cite[5.1]{RS2} for right modules.
One can use similar arguments to prove (a)-(c) and (f) for left
modules. We leave the details to the reader.

Let  $\epsilon_{n-2f}: \H_{n-2f} \rightarrow \mathscr
B_{n-2f}/\langle E_1\rangle $ be the isomorphism in (\ref{epsi}).
Let $\sigma_f: \H_{n-2f}\rightarrow \B^f/\B^{f+1}$ be the
$\kappa$-linear map defined by
$$
\sigma_f (h)= E^{f, n} \epsilon_{n-2f} (h) +\B^{f+1}.$$
Enyang~\cite[Coro. 3.4]{E} proved that $\sigma_f (h g_w)=\sigma_f(h)
T_w$ for all $h\in \H_{n-2f}$ and $w\in \mathfrak S_{n-2f}$.
Enyang~\cite{E} used $E_1E_3\cdots E_{2f-1}$ to define $\B^f$. We
use $E_{n-1}E_{n-3}\cdots E_{n-2f+1}$ to define $\B^f$. Therefore,
we have to make some modification.

It is well-known that $x_{\ss\ts^\lambda} + \H_{n-2f}^{\rhd
\lambda}$ can be considered as a $\kappa$-basis of $S^\lambda$ for
all  $\ss\in \Std_{n-2f}(\lambda)$. By abusing of notation, we write
$x_\ss=x_{\ss\t^\lambda}\in \H_{n-2f}$. Then $\epsilon_{n-2f}
(x_{\ss})\in \mathscr B_{n-2f}/\langle E_1\rangle $.

Let $\Rad S^\lambda$ be the Jacobson radical of  $S^\lambda$. Since
we are assuming that $\lambda$ is $e$-restricted,  $\Rad S^\lambda$
is the same as the radical of the invariant  form $\phi_\lambda$ on
$S^\lambda$. By abusing of notation, we denote by  $\epsilon_{n-2f}
(\Rad S^\lambda)\in  \mathscr B_{n-2f}/\langle E_1\rangle $ all
elements $\epsilon_{n-2f} (h)$ such that  $h=\sum a_{\ss} x_\ss$ and
$h+\H_{n-2f}^{\rhd \lambda}\in \Rad S^\lambda$. Similarly, we define
$\epsilon_{n-2f} ( S^\lambda)\in  \mathscr B_{n-2f}/\langle
E_1\rangle $.

Let $M$ (resp. $M_1$) be the $\kappa$-subspace  generated by
$T_u^\ast E^{f, n} \epsilon_{n-2f} (h) +\B^{\rhd (f,\lambda)}$,
$u\in D_{f,n}$ and $h\in \H_{n-2f}$ such that $\epsilon_{n-2f}(h)\in
\epsilon_{n-2f}(S^\lambda)$ (resp. $\epsilon_{n-2f}(h)\in
\epsilon_{n-2f}(\Rad S^\lambda)$). It is routine to check that
 $M\cong \Delta(f, \lambda)$ and $M_1$ is a $\B$-submodule of $M$
 such that $M/M_1\cong P_{n} (D^\lambda, \ell)$.

Let $N$ (resp. $N_1$) be the $\kappa$-subspace  generated by
$T_u^\ast E^{f+1, n+2} \epsilon_{n-2f} (h) +\mathscr B_{n+2}^{\rhd
(f,\lambda)}$, $u\in D_{f+1,n+2}$ and $h\in \H_{n-2f}$ such that
$\epsilon_{n-2f}(h)\in \epsilon_{n-2f}(S^\lambda)$ (resp.
$\epsilon_{n-2f}(h)\in \epsilon_{n-2f}(\Rad S^\lambda)$). It is
routine to check that
 $N\cong \Delta(f+1, \lambda)$ and $N_1$ is a $\mathscr B_{n+2}$-submodule of $N$
 such that $N/N_1\cong P_{n+2} (D^\lambda, \ell)$. In order to prove
 (d), we have to prove $$ \mathscr B_{n+2}
E_{n+1}\otimes_{\mathscr B_{n}} M/M_1\cong  N/N_1.$$ We remark that
$x_{\ts^\lambda} +\H_{n-2f}^{\rhd \lambda}\not\in \Rad S^\lambda$
since $ S^\lambda$ is the cyclic $\H_{n-2f}$-module generated by
$x_{\ts^\lambda} +\H_{n-2f}^{\rhd \lambda}$. This implies that
$M/M_1$ (resp. $N/N_1$)  is also the cyclic module generated by
$E^{f, n} \epsilon_{n-2f} (x_{\ts^\lambda})+M_1$ (resp. $E^{f+1,
n+2} \epsilon_{n-2f} (x_{\ts^\lambda})+N_1$).

By standard arguments, we define the  $\mathscr
B_{n+2}$-homomorphism $$\psi: \mathscr B_{n+2}
E_{n+1}\otimes_{\mathscr B_{n}} M/M_1\rightarrow N/N_1$$ such that
$$\psi(hE_{n+1}\otimes_{\B} E^{f, n} \epsilon_{n-2f} (x_\ts^\lambda)+M_1 )=h
E^{f+1, n+2} \epsilon_{n-2f} (x_\ts^\lambda) +N_1.
$$
Since $N/N_1$  is the cyclic module generated by $E^{f+1, n+2}
\epsilon_{n-2f} (x_{\ts^\lambda})+N_1$, $\psi$ is surjective. Write
$$E=E_{n-2} E_{n-4}\cdots E_{n-2f}.$$ Then $E^{f, n} =E^{f, n} E E^{f, n}$.
Therefore,$$ \mathscr B_{n+2} E_{n+1} \otimes_{\B}  M/M_1=\mathscr
B_{n+2} E^{f+1, n+2} \otimes E E^{f, n} \epsilon_{n-2f}
(x_\ts^\lambda)+M_1.$$

By \cite[2.7]{Yu} for $\mathscr B_{n+2}$, each element in $\mathscr
B_{n+2} E^{f+1, n+2}$ can be written as a linear combination of $
T_v^\ast  E^{f+1, n+2} \mathscr B_{n-2f}$ and $v\in \mathcal D_{f+1,
n+2}$ (Yu prove this result for cyclotomic Birman-Murakami-Wenzl
algebras $\mathscr B_{m, n}$ of type $G(m, 1, n)$. What we need is
the  special result for  $m=1$). So,
$$
\dim G(P_{n} (D^\lambda, \ell))\le \dim  P_{n+2} (D^\lambda, \ell),
$$ forcing $\psi$ to be injective.
 This proves (d).
Finally,  (e) follows from (a) and (d).
\end{proof}

\begin{Prop}\label{B2}  Suppose
 $e\le n-2$. If  $r,q$ are singular, then $r =\pm q^a$ for some
$a\in \mathbb Z$.
\end{Prop}

\begin{proof}
If $r, q$ are singular, then  $N_{\phi_f}(D^\lambda,\ell)\neq 0$ for
some positive integer  $f\le \lfloor n/2\rfloor$,  some irreducible
$\H_{n-2f}$-module $D^\lambda$ and some basis element $v_\ell\in
V_f$. In fact, the singularity of $r$ and $q$ are independent of
$v_\ell$. By Proposition~\ref{tool}(b), $\Rad \Delta(f, \lambda)\neq
0$. Let $D^{\ell, \mu}$ be a composition factor of
$N_{\phi_f}(D^\lambda,\ell)\neq 0$. Then $(\ell, \mu)<(f, \lambda)$.
Further,  there is a  submodule $M$ of $P(D^\lambda,\ell)$ such that
$\text{Hom}(\Delta(\ell,\mu), P(D^\lambda,\ell)/M)\neq 0$.

Suppose  $\ell=f$.   Acting  the exact functor $\F$  on
$\Delta(f,\mu)$ and $P(D^\lambda,\ell)/M$ repeatedly and using
Lemma~\ref{functor}, we  get a non-zero homomorphism from
$\Delta(0,\mu)$ to $D^\lambda/F^f(M)$. Since we are assuming that
$D^\lambda$ is an irreducible $\H_{n-2f}$-module, we have
$F^f(M)=0$. In other words, there is a non-zero epimorphism from
$\Delta(0,\mu)$ to $D^\lambda$, forcing $\mu\ge \lambda$. This
contradicts $(\ell, \mu)<(f, \lambda)$. Therefore,  $\ell < f$.

We use induction on $n$ to prove the result. We have $n\ge 4$ since
$2\le e\le n-2$. The case $n=4$ can be verified directly.  Suppose
$n>4$. By assumption, we have a non-zero homomorphism from $D^{\ell,
\mu}$ to $\Delta(f, \lambda)/N$ for some submodule $N\subset
\Delta(f, \lambda)$. Acting the exact functor $\F$ and using
Lemma~\ref{functor}, we can assume $\ell=0$.

Suppose  $f=1$. Enyang~\cite{E} defined the Jucys-Murphy elements
$L_i$, $1\le i\le n$ for $\B$ such that $L_1=r$ and  $L_i=T_{i-1}
L_{i-1} T_{i-1}$. Further, he proved that $L_1, \cdots, L_n$ commute
each other and  $\prod_{i=1}^n L_i$ is a central element of $\B$.
Therefore, $\prod_{i=1}^n  L_i$ acts on each cell module (and hence
its non-zero quotient modules) as a scalar. We use it to act on both
$D^{0, \mu}$ to $\Delta(1, \lambda)/N$. By \cite[Theorem~2.2]{RS2},
$$\prod_{p\in [\lambda]} c(p)=  \prod_{p\in [
\mu]} c(p) ,$$ where $c(p)$ is defined in (\ref{lamcont}).
Therefore,  $r=\pm q^a$ for some integer $a$, as required.

Suppose $f>1$. As $\mathscr B_{n-1}$-module, $D^{0, \mu}$ may be
reducible. Further, each composition factor is of form $D^{0, \eta}$
for some partition $\eta$ with $|\eta|=|\mu|-1$.  As $\mathscr
B_{n-1}$-module, $\Delta(f, \lambda)$ has
$\Delta$-filtration~\cite[Coro. 5.8]{E}. By Lemma~\ref{hom}, $D^{0,
\eta}$ has to be a composition factor of $\Delta(f_1, \tilde
\lambda)$ for some suitable $(f_1, \tilde\lambda)\in \Lambda_{n-1}$
such that $f_1$ is either $f$ or $f-1$. In particular, $f_1\neq 0$.
Further, in the first case, $\tilde \lambda$ can be obtained from
$\lambda$ by adding an addable node. In the second case, $\tilde
\lambda$ can be obtained from $\lambda$ by adding an addable node.
By induction assumption on $n-1$, $r= \pm q^a$ for some $a\in
\mathbb Z$.
\end{proof}

\begin{Prop}\label{B3}
 If $e\le n-2$ and $r=\pm q^a$ for some integer $a$, then $r$ and $q$ are singular.
\end{Prop}

\begin{proof}
Since $e\le n-2$, we can assume $r=\pm q^a$ for some non-negative
integer $a$ with $a<e$.  What we want to do is to find some suitable
$\lambda\in \Lambda^+(n-2f)$ with $f>0$ such that
$S^\lambda=D^\lambda$. We remark that this can be verified by
Theorem~\ref{BMW1} for $f=0$. We will define another $(\ell, \mu)\in
\Lambda_n$. One can use the Definition~\ref{adm} to verify that
$\mu$ is $(f-\ell, \lambda)$-admissible. By Theorem~\ref{BMW1} and
Proposition~\ref{tool},
 $\det G_{f,\lambda}=0$ and $N_{\phi_f}(D^\lambda,\ell)\cong \Rad
\Delta(f,\lambda)\neq 0$ for some $\ell, 1\le \ell\le \dim V_f$. So,
$r$ and $ q$ are singular.

In the remaining part of this proof, we will construct $(\ell, \mu)$
and $(f, \lambda)$ explicitly. We define  $b=a+1 $. So,
 $1\le b\le e\le n-2$. Since $o(q^2)=e$, we have $q^e\in \{-1,
 1\}$. We will use this fact frequently in the remaining part of the
 proof.

\Case{1.   $r=\pm q^a$ and $r\not\in \{q^{-1},-q\}$}

We define $(f, \lambda)=(\frac {n-b}{2}, (1^b)), \text{ and } (\ell,
\mu)=(\frac {n-b-2}{2}, (2, 1^b)) $ if  $ n-b$ is even.
 Otherwise, there are
several subcases we have to discuss.

 \textbf{subcase 1. $b\not\in\{e-1, e-2\}$}

 If $b\neq n-3$, then
$b<n-3$. Otherwise, $b=e=n-2$ forcing $2\mid n-b$, a contradiction.
Therefore, $\frac{n-b-3}{2}\ge 1$. We define
 $(\ell, \mu)=(\frac {n-b-5}{2}, (3, 2, 1^b))$,   and $ (f, \lambda)= (\frac
{n-b-3}{2}, (2, 2, 1^{b-1})$.

If $b=n-3$, we have $e=b=n-3$ and  $r=\pm q^a=\pm q^{-1}$.
Otherwise, $e= n-2$   forcing $b= e-1$, a contradiction.  If
$\Char(\kappa)=2$, there is nothing to be proved. Otherwise, since
we are assuming $r\neq q^{-1}$, we have $r=-q^{-1}$.

If $n$ is even, then $b$ and $e$ have to be odd. So, $e>2$. If $n$
is odd, then $b$ is even forcing $b\ge 2$ and $n\ge 5$. We define
$$ (\ell, \mu)=\begin{cases} (\frac{n-4}{2}, (3,1)), & \text{ if $2\mid n$}, \\
(\frac{n-5}{2}, (2^2, 1)), &\text{otherwise.}\\
\end{cases}$$
 and
$$ (f, \lambda)=\begin{cases} (\frac{n-2}{2}, (2)), & \text{ if $2\mid n$}, \\
(\frac{n-3}{2}, (1^3)), &\text{otherwise.}\\
\end{cases}$$

\textbf{subcase 2. $b=e-2$.} We have $r=\pm q^{-3}$, $e\ge 3$ and
$n\ge 5$.

If $n$ is even, then $n\ge 6$ and $2\nmid e$. We define

$$ (\ell, \mu)=\begin{cases} (\frac{n-6}{2}, (5, 1)), & \text{if $e>4$}, \\
 (\frac{n-3}2, (2,1)) &\text{if $e=3$,}\\
\end{cases}$$
and
$$ (f, \lambda)=\begin{cases} (\frac{n-4}2, (4)), & \text{if $e>4$}, \\
(\frac{n-1}2, (1)) &\text{if $e=3$.}\\
\end{cases}$$

If $n$ is odd, then $n\ge 7$. Otherwise, $n=5$ and $e=3$, which
contradicts $2\nmid n-b$. We define
$$ (\ell, \mu)=\begin{cases} (\frac{n-7}2, (4,2,1))   & \text{if $e\neq 5$}, \\
(\frac{n-5}2, (2,1^3)) &\text{if $e=5$,}\\
\end{cases}$$
and
$$ (f, \lambda)=\begin{cases} (\frac{n-5}2, (3,1^2)) , & \text{if $e\neq 5$}, \\
(\frac{n-3}2, (1^3)) &\text{if $e=5$.}\\
\end{cases}$$

\textbf{subcase 3. $b=e-1$.} We have $r=\pm q^{-2}$, $e\ge 2$ and
$n\ge 4$.

Suppose $n$ is odd. We have $2\nmid e$ and $n\ge 5$. We define
$$ (\ell, \mu)=\begin{cases}  (\frac{n-5}2, (3,2))   & \text{if $e\neq 3$}, \\
(\frac{n-5}2, (2^2, 1)) &\text{if $e=3$,}\\
\end{cases}$$
and
$$ (f, \lambda)=\begin{cases}  (\frac{n-3}2, (2, 1)), & \text{if $e\neq 3$}, \\
 (\frac{n-1}{2}, (1)), &\text{if $e=3$.}\\
\end{cases}$$

Suppose  $n$ is even. We define
$$ (\ell, \mu)=\begin{cases}  (\frac{n-6}2, (3^2))   & \text{if $n\ge 6$}, \\
(0, (2^2)) &\text{if $n=4$,}\\
\end{cases}$$
and
$$ (f, \lambda)=\begin{cases}  (\frac{n-2}2, (1^2)), & \text{if $n\ge  6$}, \\
 ({2}, (0)), &\text{if $n=4$.}\\
\end{cases}$$

This completes the proof of our result for $r=\pm q^a$ and $r\not\in
\{q^{-1}, -q\}$.

 \Case{2. $r\in \{q^{-1},-q\}$}

\textbf{subcase 1. $e=2$.} Then  $n\ge 4$ and  $r\in \{q, -q\}$.

If $n$ is even, we define $(\ell, \mu)=(\frac{n-4}2, (2, 1^2))$ and
$(f, \lambda)=(\frac{n-2}{2}, (1^2))$.

If $n$ is odd, then $n\ge 5$. We define $(\ell, \mu)=(\frac{n-5}2,
(2^2, 1))$  and $(f, \lambda)=(\frac{n-1}{2}, (1))$.

\textbf{subcase 2. $e>2$.} Then  $n\ge 5$.

Suppose $n$ is even. Then $n\ge 6$. We define
$$ (\ell, \mu)=\begin{cases}  (\frac{n-4}2, (3, 1)),  & \text{if $r=q^{-1}$}, \\
 (\frac{n-4}2, (2, 1^2)), &\text{if $r=-q$,}\\
\end{cases}$$
and
$$ (f, \lambda)=\begin{cases}  (\frac{n-2}2, (2)), & \text{if $r=q^{-1}$}, \\
 (\frac{n-2}{2}, (1^2)), &\text{if $r=-q$.}\\
\end{cases}$$

Suppose $n$ is odd and $r=q^{-1}$. We define
$$(\ell, \mu)=\begin{cases} (\frac{n-7}2, (3^2, 1)), & \text{if $n\ge 7, e\neq
5$, }\\
(\frac{n-7}2, (3, 2^2)), & \text{if $n\ge 7, e=
5$, }\\
(0, (2, 1^3)), & \text{if $n=5, e=
3$, }\\
\end{cases}
$$
and
$$(f, \lambda)=\begin{cases} (\frac{n-5}2, (3, 1^2)), & \text{if $n\ge 7$ and
$e\neq
5$, }\\
(\frac{n-5}2, (2^2, 1)), & \text{if $n\ge 7$ and $e=
5$, }\\
(1, (1^3)), & \text{if $n=5$ and $e=
3$.}\\
\end{cases}
$$

Suppose $r=-q$. If $n$ is even, we define $(\ell, \mu)=(\frac{n-4}2,
(2, 1^2))$ and $(f, \lambda)=(\frac{n-2}{2}, (1^2))$.

If $n$ is odd and $n\ge 7$, we define
$$ (\ell, \mu)= \begin{cases} (\frac{n-7}2, (3, 2^2)), &\text{if $e\neq 5$, }\\
(\frac{n-7}2, (3^2, 1)), &\text{if $e=5$,}\\
\end{cases}$$
and
$$(f, \lambda)=\begin{cases} (\frac{n-5}2, (3, 1^2)), &\text{if $e\neq 5$, }\\
(\frac{n-5}2, (3, 2)), &\text{if $e=5$.}\\
\end{cases}$$

If $n=5$, then $e=3$. We define $(\ell, \mu)=(0, ( 1^5))$ and $(f,
\lambda)=(1, (1^3))$. This completes the proof of our result for
$r\in \{q^{-1}, -q\}$.
\end{proof}

\noindent \textbf{Proof of Theorem~\ref{main}:} Theorem~\ref{main}
follows from Propositions~\ref{B1}, \ref{B2}--\ref{B3}, immediately.

\providecommand{\bysame}{\leavevmode ---\ } \providecommand{\og}{``}
\providecommand{\fg}{''} \providecommand{\smfandname}{and}
\providecommand{\smfedsname}{\'eds.}
\providecommand{\smfedname}{\'ed.}
\providecommand{\smfmastersthesisname}{M\'emoire}
\providecommand{\smfphdthesisname}{Th\`ese}


\begin{thebibliography}{DWH99}

\bibitem{BW}
{\scshape J.~S. Birman {\normalfont \smfandname} H.~Wenzl}, {\og
Braids, link
  polynomials and a new algebra\fg}, \emph{Trans. Amer. Math. Soc.}
  \textbf{313} (1989), 249--273.

\bibitem{B}
{\scshape R.~Brauer}, {\og On algebras which are connected with the
semisimple
  continuous groups\fg}, \emph{Ann. of Math.} \textbf{38} (1937), 857--872.

\bibitem{CMW}
{\scshape A. Cox, M. De Visscher and P. Martin}, {\og
 The block of the Brauer algebra in characteristic zero\fg},
 \emph{Representtaion Theory},  \textbf{13} (2009), 272--308.


\bibitem{DWH}
{\scshape W.~F. Doran, IV, D.~B. Wales {\normalfont \smfandname}
P.~J. Hanlon},
  {\og On the semisimplicity of the {B}rauer centralizer algebras\fg}, \emph{J.
  Algebra} \textbf{211}, (1999), 647--685.

\bibitem{Enyang}
{\scshape J.~Enyang}, {\og Cellular bases for the Brauer and
Birman-Murakami-Wenzl algebras\fg}, \emph{J. Algebra}, \textbf{281}
(2004), 413-449.

\bibitem{E}\bysame, {\og Specht modules and semisimplicity criteria for Brauer and
Birman-Murakami-Wenzl algebras\fg}, \emph{J. Alg. Comb.},
\textbf{26}, (2007), 291-341.

\bibitem{GL}
{\scshape J.~J. Graham {\normalfont \smfandname} G.~I. Lehrer}, {\og
Cellular
  algebras\fg}, \emph{Invent. Math.} \textbf{123} (1996), 1--34.

\bibitem{K}
{L., Kauffman}, ``An invariant of regular isotopy",
  \textit{Trans. Amer. Math. Soc.} 318 (1990): 417--471.


\bibitem{KX}
{\scshape S. Koenig and C.C. Xi}, {\og A characteristic free
approach to Brauer algebras\fg }, \emph{Trans. Amer. Math.
Soc.}\textbf{ 353} (2001), no. 4, 1489--1505.

\bibitem{MW}{H.R.  Morton and A.J.  Wassermann},
`` A basis for the Birman-Murakami-Wenzl algebra", unpublished
paper, 2000.

\bibitem{Mu}{\scshape J. Murakami} , {\og The Kauffman polynomial of links and
representaion theory\fg}, \textit{Osaka J. Math.} \text{24}, (1987),
745--758.

\bibitem{Na}{\scshape M.~Nazarov}, {\og Young's orthogonal form for {B}rauer's
centralizer algebra\fg}, \emph{J. Algebra} \textbf{182} (1996),
664--693.


\bibitem{RS3}{\scshape H. Rui {\normalfont \smfandname} M. Si} {\og Discriminants for Brauer algebras \fg }, \emph {Math. Z.}\textbf  {258}, no. 4 (2008), 925--944.

\bibitem{RS:bmw}{\scshape H. Rui {\normalfont \smfandname} M. Si}, {\og Gram determinants and semisimplicity criteria for
Birman-Wenzl algebras\fg},\emph{J. Reine. Angew. Math. },
\textbf{631}, (2009) 153-180.

\bibitem{RS2}{\scshape H. Rui {\normalfont \smfandname} M. Si}, {\og Blocks of
Birman-Murakami-Wenzl algebras\fg}, \emph{Int. Math. Res. Notes},
Article no. rnq083, (2010), 35 pages

\bibitem{RS4}{\scshape H. Rui {\normalfont \smfandname} M. Si}, {\og Non-vanishing Gram determinants for cyclotomic NW and BMW algebras},
\emph{prepint}, 2010.

\bibitem{Gap}{\scshape H. Rui {\normalfont \smfandname} M. Si},  GAP program on discriminants for Birman-Wenzl algebras,  available at http://math.ecnu.edu.cn/~hbrui/bmw.g

\bibitem{W2}
{\scshape H. Wenzl} {\og Quantum groups and subfactors of type $B,
C,$ and $D$\fg}, \emph{Comm. Math. Phys.} \textbf{133} (1988),
383--432.



\bibitem{Xi}{\scshape C.C. Xi}, {\og On the quasi--heredity of Birman--Wenzl algebras\fg}, \emph{ Adv. Math.} \textbf{154} (2) (2000)
280--298.

\bibitem{Yu} {\scshape S. Yu}, {\og The cyclotomic Birman-Murakami-Wenzl algebras\fg},
\emph{Ph.D thesis},  University of Sydney, 2007.
\end{thebibliography}
\end{document}